\documentclass[ECP]{ejpecp} 
\usepackage[T1]{fontenc}
\usepackage[utf8]{inputenc}
\usepackage{enumerate}  
\usepackage{ulem}
\usepackage{color}

\SHORTTITLE{LD for the greedy exploration process on configuration models}
\TITLE{Large deviations for the greedy exploration process on configuration models\thanks{Supported
    by CICADA-UdelaR}} 

\AUTHORS{
Bermolen~Paola \footnote{Universidad de la Rep\'ublica, Uruguay. \EMAIL{paola@fing.edu.uy}}
  \and 
Goicoechea~Valeria \footnote{Universidad de la Rep\'ublica, Uruguay.
    \EMAIL{vgoicoechea@fing.edu.uy}}
 
  \and 
 Jonckheere~Matthieu \footnote{LAAS CNRS, France. \EMAIL{matthieu.jonckheere@laas.fr}}
  }


\VOLUME{0}
\YEAR{2020}
\PAPERNUM{0}
\DOI{10.1214/YY-TN}

\ABSTRACT{
We prove a large deviation principle for the greedy exploration of configuration models, building on a time-discretized version of the method proposed by \cite{Bermolen2017} and \cite{Brightwell-Janson-Luczak} for jointly constructing a random graph from a given degree sequence and its exploration.
The proof of this result follows the general strategy to study large deviations of processes proposed by \cite{F&K}, based on the convergence of non-linear semigroups. 
We provide an intuitive interpretation of the LD cost function using Cr\'amer's theorem for the average of random variables with appropriate distribution, depending on the degree distribution of explored nodes.
The rate function can be expressed in a closed-form formula, and
the large deviations trajectories can be obtained through
explicit associated optimization problems. We then deduce large deviations results for the size of the independent set constructed by the algorithm. As a particular case, we analyze these results for $d$-regular graphs. 
}

\newcommand{\R}{\mathbb{R}}

\newcommand{\PP}{\mathbb{P}}                                                              

\begin{document}
\section{Introduction}
In this paper, we analyze a simple greedy algorithm to construct an independent set over a random graph chosen uniformly from those with a given degree. We use a simultaneous construction of 
the random graph from a given degree sequence (i.e., it is a \emph{configuration model}) and an exploration discovering independent set. This idea was first used by \cite{wormaldDF} for $d$-regular graphs and then for \cite{Bermolen2017} and \cite{Brightwell-Janson-Luczak} for more general uniform random graphs. We consider a time-discretized version of the algorithm proposed by \cite{Brightwell-Janson-Luczak} for a bounded degree sequence.

We prove a large deviation principle (LDP), when the number $N$ of the graph's vertices goes to infinity, for a rescaled version $X_t^N=  \frac{X_{[Nt]}^N}{N}$ ($t\in [0,1]$) of the multidimensional Markov chain $\left\{X_n^N\right\}_n$ that counts the number of vertices that have already been assigned into the independent set, and the number of \emph{empty} (or non-explored) vertices from each degree at each step $n$ of the algorithm.


The proof of this result is based on the general strategy proposed by Feng and Kurtz in \cite{F&K}, which is based on the convergence of non-linear semigroups. Their seminal work consists of combining the tools of probability, analysis, and control theory used in the works of \cite{deAcosta}, \cite{Dupuis}, \cite{Evans_Ishii}, \cite{Fleming}, \cite{Puhalskii}, and others to propose a general strategy for the study of large deviations of processes. In the case of Markov processes, this program is carried out in four steps. The first step consists of proving the convergence of non-linear generators $\mathbf{H}_N$ and derives the limit operator $\mathbf{H}$. The second step consists of verifying the \emph{exponential compact containment condition}. The third step consists of proving that $\mathbf{H}$ generates a semigroup $\mathbf{V}= \left\{V_t\right\}_t$. 
The rate function is constructed in terms of that limit $\mathbf{V}$. This limit semigroup usually admits a variational form known as \emph{Nisio semigroup} in control theory. Then, the fourth step consists in constructing a version of the rate $I$ that is much more tractable in practice. 

In our case, after working on the four steps mentioned before, we prove that the rate function can be expressed as an action integral. Moreover, its cost function has an intuitive interpretation in terms of Cr\'amer's theorem for the average of random variables with appropriate distribution, depending on the degree distribution of explored vertices in each step of the algorithm.

We provide a way to find the trajectory that minimizes the LD rate function over a set of trajectories (i.e., the most probable trajectory) by studying the Hamiltonian dynamics associated with the rate function obtained and deduce LD results for the size of the independent set constructed by such an algorithm.  

The rest of the paper is organized as follows. 
In Section \ref{section: Description}, we define the dynamic analyzed in this article, which consists of simultaneously constructing a random graph and an independent set from an initial distribution of degrees. Moreover, we define a sequence of Markov processes related to this algorithm. 
In Section \ref{section: Main-results}, we present the main result of this article: a path-state LDP for the sequence of Markov processes defined in Section \ref{section: Description} along with the heuristic that motivates the result. The detailed proof is deferred to Section \ref{Proof_LDP}.
As a corollary, we obtain the fluid limit of the process and LD results for the size of the independent set constructed by the algorithm. 
In Section \ref{section: d-regular}, we analyze those results for the particular case of $d$-regular graphs. 

\section{Description of the dynamics }\label{section: Description}
This section presents the dynamics considered in this article, which consist of a simultaneous construction of a random graph and an independent set from an initial distribution of degrees.

We start with a set of vertices $\mathcal{V}^N= \left\{1, 2, \dots, N\right\}$, such that $\text{deg}(i)\leq D <\infty$ for all $i$ and such that the initial distribution of degrees
$\frac{1}{N} \# \left\{i:\, \text{deg}(i)=j\right\}$ converges to $p_j\geq 0$, when the number of vertices $N$ goes to infinity, for all $j=0,\dots D$ with $\sum\limits_{j=0}^D p_j=1$. Let us denote $\lambda= \sum_j j p_j$. 

Each vertex $i$ of the graph has a number $\text{deg} (i)$ of half-edges available to be paired with the half-edges of other vertices. Next we describe how these half-edges are paired as the random graph is constructed.

At each step $n=0, 1, \dots, N$, the set $\mathcal{V}^N$ is partitioned into three classes:
\begin{itemize}
\item a set $\mathcal{S}_n^N$ of vertices that have already been assigned into the independent set, with all half-edges paired with vertices out of $\mathcal{S}_n^N$;
\item a set $\mathcal{B}_n^N$ of \emph{blocked vertices}, where at least one of its half-edges has been paired with a half-edge from $\mathcal{S}_n^N$;
\item a set $\mathcal{E}_n^N$ of \emph{empty vertices}, from which no half-edge has yet been paired. $\mathcal{E}_n^N$ can be decomposed as $\mathcal{E}_n^N= \bigcup\limits_{j=0}^D \mathcal{E}_n^N(j)$, where $\mathcal{E}_n^N(j)$ is the set of empty vertices of degree $j$ at step $n$.
\end{itemize}

Initially, all vertices are \emph{empty}, i.e. $\mathcal{E}_0^N=\mathcal{V}^N$ and $\mathcal{S}_0^N= \emptyset$. At step $n$, a vertex $v$ is selected uniformly from $\mathcal{E}_n^N$, it is assigned to $\mathcal{S}_n^N$, and all its half-edges are paired, drawing uniformly within the available half-edges. This pairing procedure results in the following updates:

\begin{itemize}
\item $v$ is moved from $\mathcal{E}_n^N$ to $\mathcal{S}_n^N$,
\item each half-edge incident to $v$ (if it has some) is paired  with some other uniformly randomly chosen vertices among the currently unpaired half-edges,
\item all vertices in $\mathcal{E}_n^N$ with some half-edges already paired with a half-edge from $v$ are moved to $\mathcal{B}_n^N$.
\end{itemize}

Note that some half-edges from $v$ may be paired with half-edges from $\mathcal{B}_n^N$, or indeed with other half-edges from $v$, and no change in the status of a vertex results from such pairings. At each step $n$, the only paired edges are those with at least one endpoint in $\mathcal{S}_n^N$. This is the main difference between the dynamics described in \cite{Bermolen2017} and \cite{Brightwell-Janson-Luczak} for a continuous-time version of this algorithm. In \cite{Bermolen2017}, the neighbors of blocked vertices are revealed at each step, meaning that degrees of empty vertices can change over time. For simplicity, we do not do this.

The algorithm terminates at the first step $n=T_N^*$ at which $\mathcal{E}_n^N=\emptyset$. At this point, there may still be some unpaired half-edges pointing out from blocked vertices. These may be paired off uniformly at random to complete the construction of the graph $G\left(N, \left(\text{deg}(1), \dots, \text{deg}(N)\right)\right)$. Note that $T_N^*$ coincides with the size of the independent set constructed by the algorithm. The expected value of $\frac{T_N^*}{N}$ is usually called the \emph{jamming constant} of the graph.

For each $n\in \left\{0,1,\dots, N \right\}$, let us denote $X^N_n= \left(S^N_n, U^N_n, E^N_n(0), E^N_n(1), \dots, E^N_n(D) \right)$ with:
\begin{itemize}
\item $S_n^N=\left|\mathcal{S}_n^N\right|$, the number of vertices that have already been assigned into the independent set at step $n$;
\item $U_n^N$, the total number of unpaired half-edges (corresponding to empty or blocked vertices) at step $n$;
\item $E_n^N(j)=\left|\mathcal{E}_n^N(j)\right|$, the number of empty vertices with degree $j$ at step $n$.
\end{itemize}
$\left\{X^N_n\right\}_n$ is a discrete-time Markov process in $\R^{D+3}$. By construction, it is updated at step $n+1$ as follows: 
\begin{itemize}
\item The vertex $v$ is assigned to $\mathcal{S}_n^N$. Then, $S_{n+1}^N=S_n^N +1$.

\item If $v \in \mathcal{E}_n^N(k)$ with $k\neq 0$, then:

\begin{enumerate}
\item Each one of the $k$ half-edges pointing out from $v$ is paired in turn with some other uniformly randomly chosen between the currently unpaired half-edges. Let $H^N$ be the number of half-edges from $v$ that are paired with another vertex different from $v$ (blocked or empty), i.e., that do not form loops. Note that $H^N$ has a Hypergeometric distribution $\mathbf{Hyper} \left(U_n^N, U_n^N-k, k\right)$.
Then, $U_{n+1}^N=U_n^N -k-H^N$.
\item We have to distribute those $H^N$ half-edges between the $U_n^N-\sum_{j} jE_n^N(j)$ half-edges corresponding to blocked vertices and the $\sum_{j} jE_n^N(j)$ half-edges corresponding to empties. Let $B^N$ be the number of half-edges of $v$ that are paired to blocked vertices, then $B^N$ has a Hypergeometric distribution with parameters $U_n^N-k$, $U_n^N-\sum_{j=1}^D j E_n^N(j)$ and $H^N$.

\item Now, if $H^N=h$ (with $h\leq k$) and $B^N=b$ (with $b\leq h$), there are $h-b$ half-edges to distribute among the empties. Let $W_j^N$ be the number of half-edges from $v$ that are paired to some $w\in \mathcal{E}_n^N(j)$. Note that $\left(W_1^N, \dots, W_D^N\right)$ has a (multivariate) Hypergeometric distribution:\\
$
 \mathbf{Hyper}\left( \sum_{j} jE_n^N(j)-k, E_n^N(1), \dots, k\left(E_n^N(k)-1\right), \dots, DE_n^N(D), h-b\right)
$.
\item Finally, let $\tilde{W}_j^N$ be the number of empty vertices of degree $j$ that share at least one edge with $v$.  Then, $E_{n+1}^N(0)=E_n^N(0)$, $E^N_{n+1}(j)=E_n^N(j)-\tilde{W}_j^N$ if $j\neq k$ and $E^N_{n+1}(k)=E_n^N(k)-1-\tilde{W}_k^N$. 
\end{enumerate}
\item If $\mathbf{deg}(v)=0$, then $S_{n+1}^N=S_n^N+1$, $U_{n+1}^N = U_n^N$, $E_{n+1}^N(0)=E_n^N(0)-1$ and $E_{n+1}^N(j)=E_n^N(j)$ for all $j\neq 0$.
\end{itemize}

According to the following Lemma, the distribution of $\left( \tilde{W}_1, \dots, \tilde{W}_D\right)$ can be approximated by the Hypergeometric distribution corresponding to $\left( W_1, \dots, W_D\right)$.

\begin{lemma}\label{Lemma:W} Let $x=\left(s, u, e_0, \dots, e_D\right)$ be an element in the space state of $\left\{X_n^N\right\}_n$, and $\left(\omega_j\right)_j$ with $0\leq \omega_j\leq e_j$ such that $\sum_j \omega_j \leq h-b$. Then,
\begin{gather*}
\underset{N\rightarrow \infty}{\lim} \mathbb{P}\left(\left(\tilde{W}_j^N\right)_j = \left( \omega_j\right)_j \Big\vert X_n^N=x; \mathbf{deg}(v)=k;  H^N=h; B^N=b; \left(W_j^N\right)_j=\left(\omega_j\right)_j \right)=1
\end{gather*}
\end{lemma}
\begin{proof}
See  Equation \textbf{17} from \cite{Bermolen2017}. In the article notation: $W_j^N= Y\left(\mu_{t^-}\right)(j)$ and $\tilde{W}_j^N= \tilde{Y}\left(\mu_{t^-}\right)(j)$.
\end{proof}

Let $X_t^N:=\frac{X_{[Nt]}^N}{N}$ be a rescaled version of $X_n^N$ with $t \in [0,1]$. The state space of  $X_t^N$ is  
$
E^N= \left\{ \frac{1}{N}(\hat{s}, \hat{u}, \hat{e}_0, \dots, \hat{e}_D):\, \hat{s}, \hat{e}_i\in \{0,\dots, N\}; \, \sum_{j} j\hat{e}_j \leq \hat{u}\right\}
$
which is a subset of the compact set $E:=\left\{(s,u, e_0, e_1,\dots, e_D) \in [0,1]\times \R \times[0,1]^{D+1}:\,  \sum_{j} je_j \leq u\leq \lambda\right\}$.  The size of the independent set constructed by such an algorithm is given by
$
T_N^* = \inf \left\{ n:\, \sum_{j} E_n^N(j)=0\right\} = N \inf \left\{ t\in [0,1]: \,  \sum_{j} E_{[Nt]}^N(j) =0 \right\}
$.

We provide large deviations for both sequences $\left\{X^N_.\right\}_N$ and $\left\{\frac{T_N^*}{N}\right\}_N$.

\section{Main results} \label{section: Main-results}
In this section, we present the main results of the paper. 
In Subsection \ref{subsection:1}, we present an LDP for $X^N = \left\{X^N_t\right\}_{0\leq t \leq 1}$  and a heuristic description to derive this result. The proof of this theorem is based on the work of \cite{F&K} and is deferred to Section \ref{Proof_LDP}. 
In Subsection \ref{subsection:2}, we deduce the corresponding fluid limit.  
In Subsection \ref{subsection:3}, we provide a way of finding the trajectory that minimizes the LD rate function over a set of trajectories (i.e., the most probable trajectory) by studying the Hamiltonian dynamics associated with the rate function obtained.  
Finally, in Subsection \ref{subsection: 4}, we deduce an LD result for the size of the independent set constructed by such an algorithm. 

\subsection{LDP for $\left\{X^N\right\}_N$}  \label{subsection:1}
We now state our main result.
\begin{theorem}[LDP for $\left\{X^N\right\}_N$] \label{thm:LDPX^N}
The sequence $\left\{ X^N \right\}_N$  with $X^N = \left\{X^N_t\right\}_{0\leq t \leq 1}$ verifies an LDP on $D_E\left[0,1\right]$ with good rate function $I\colon D_E [0,1] \rightarrow \left[0, +\infty\right]$ such that $I\left( \mathbf{x}\right)= \intop_{0}^{1} L \left( \mathbf{x}(t), \dot{\mathbf{x}}(t) \right) \text{d}t$ if $\mathbf{x} \in \mathcal{H}_L$ and it is $+\infty$ in other case. 
$L\colon E\times \R^{D+3}\rightarrow \mathbb{R}$ is the cost function
\begin{equation} \label{eq:LX}
L(x, \beta)= \sup_{\alpha \in \R^{D+3}} \left\{ \left\langle \alpha, \beta\right\rangle - H\left( x, \alpha\right)\right\},
\end{equation}
with $H\colon E \times \R^{D+3}\rightarrow \R$ given by
\begin{equation} \label{eq:HX}
H\left(x, \alpha\right)= \begin{cases}
\log \left[  {\displaystyle \sum_{k}} e^{\alpha_s-2k\alpha_u-\alpha_k} \left(1+\sum\limits_{j=1}^D \left(e^{-\alpha_j}-1\right) \frac{je_j}{u}\right)^k \frac{e_k}{ \sum_{j} e_j} \right], & \text{ if } \sum_{j} e_j>0, \\
0, & \text{ if }  \sum_{j} e_j=0, 
\end{cases}
\end{equation}
being $x=\left(s, u, e_0, e_1, \dots, e_D\right)$ and $\alpha=\left(\alpha_s, \alpha_u, \alpha_0, \alpha_1, \dots, \alpha_D\right)$. $\mathcal{H}_L$ is the set of all absolutely continuous function $\mathbf{x}: [0,1]\rightarrow E$, $\mathbf{x}(t)=\left(s(t), u(t), e_0(t), e_1(t), \dots, e_D(t)\right)$ with initial value $\mathbf{x}(0)=(0, \lambda, p_0, p_1, \dots, p_D)$ and such that $s(t)$ is increasing, $u(t)$ and $e_j(t)$ are decreasing, and the integral $\int_0^1 L\left(\mathbf{x}(t), \dot{\mathbf{x}}(t)\right)dt$ exists and it is finite.
\end{theorem}

The proof of this result is deferred to Section \ref{Proof_LDP}. In what follows we provide an intuitive way to construct the cost function of the LDP in terms of the rate function provided by Cramer's theorem for the average of the approximated distribution of the new explored vertices in one step conditioning to the number of explored vertices.

Consider a curve $\mathbf{x}(t)= \left( s(t), u(t), e_0(t), \dots, e_D(t)\right)$ contained in $E$ and such that $X_t^N \approx \mathbf{x}(t)$. Then, the infinitesimal increments $\dot{\mathbf{x}}(t)$ correspond to the mean number of new explored vertices from each degree in one step, as can be deduced from the following statement:
\begin{gather*}
\dot{\mathbf{x}}(t)\approx \frac{\mathbf{x}(t+h)-\mathbf{x}(t)}{h} \approx \frac{X_{[N(t+h)]}^N - X_{[Nt]}^N}{Nh} = \displaystyle \frac{1}{Nh} \sum_{n=[Nt]}^{[N(t+h)]-1} \left(X_{n+1}^N-X_n^N\right).
\end{gather*}

\begin{proposition}
The distribution of the number of new explored vertices in one step $X_{n+1}^N-X_n^N$, conditioning to $X_t^N \approx x(t)=\left( s(t), u(t), e_0(t), \dots, e_D(t)\right)$, can be approximated by the random vector:
\begin{gather*}
Z^{\mathbf{x}(t)}=\begin{cases}
\left(1,0,-1,0,\dots,0\right), & \text{ with probability } \frac{e_0(t)}{\sum_j e_j(t)},\\
\left(1,-2k,0,-M_1, \dots, -1-M_k, \dots, -M_D\right), & \text{ with probability } \frac{e_k(t)}{\sum_j e_j(t)} \, (1\leq k\leq D),
\end{cases}
\end{gather*}
where $M\sim \mathbf{Mult}\left(K-B, q_1, \dots, q_D\right)$ is a multinomial vector depending on $K$ such that $\PP(K=k)=\frac{e_k(t)}{\sum_j e_j(t)}$ for $k\in \{0,\cdots, D\}$, $B=B(K)\sim \mathbf{Bin}\left(K, 1-\frac{\sum_j je_j(t)}{u(t)}\right)$, and $q_i=\frac{i e_i(t)}{\sum_{j}j e_j(t)}$.

\end{proposition}

\begin{proof}

At step $n$, the vertex $ v\in\mathcal{E}_n^N(k)$ is drawn uniformly. If $k\neq 0$, then $X_{n+1}^N-X_n^N=\left(1, -k-H^N, 0, -\tilde{W}_1^N, \dots, -1-\tilde{W}_k^N, \dots, -\tilde{W}_D^N\right)$, where

\begin{itemize}
\item $H^N$ is the number of half-edges from $v$ that are joined to another node different from $v$. As the probability of loops converges to zero (see \cite{Brightwell-Janson-Luczak}), then $H^N\approx k$.

\item Lemma \ref{Lemma:W} assures that $\tilde{W}_j^N \approx W_j^N$, where $\left(W_1^N, \dots, W_D^N\right)$ has a (multivariate) Hypergeometric distribution with parameters
$\sum_j jE_n^N(j)-k$, $ E_n^N(1)$, $\dots$, $jE_n^N(j)$, $\dots$, $k\left(E_n^N(k)-1\right)$, $\dots$, $DE_n^N(D)$, and $k-B^N$. Note that $B^N$ can be approximated by a Binomial random variable $B$ with parameters $n=k$ and $ p=\lim_N \frac{U_n^N-\sum_j jE_n^N(j)}{U_n^N-k}=1-\frac{\sum_j je_j(t)}{u(t)}$. Moreover, $\left(W_1^N, \dots, W_D^N\right)$ can be approximated by a Multinomial random vector
$
\left(W_1^N, \dots, W_D^N\right) \approx \left(M_1, \dots, M_D\right) \sim \mathbf{Mult}\left(k-B, q_1, \dots, q_D\right)
$,
with $q_i=\frac{i e_i(t)}{\sum_{j}j e_j(t)}$.
\end{itemize}

If $v\in \mathcal{E}_n^N(0)$, then $X_{n+1}^N-X_n^N=\left(1,0,-1,0,\dots,0\right)$. 

\end{proof}

\begin{proposition}
The cost function $L\left(\mathbf{x}(t), \dot{\mathbf{x}}(t)\right)$ defined in \eqref{eq:LX} coincides with 
the LDP rate function for the average of  i.i.d random variables $(Z_i^{\mathbf{x}(t)})_{i \in  \mathbb N}$ distributed as $Z^{\mathbf{x}(t)}$:
\begin{gather*}
L\left(\mathbf{x}(t), \dot{\mathbf{x}}(t)\right)=\underset{\alpha}{\sup} \left\{\left\langle \alpha, \dot{\mathbf{x}}(t)\right\rangle - H\left(\mathbf{x}(t), \alpha\right)\right\}= \Lambda^*_{Z^{\mathbf{x}(t)}}\left(\dot{\mathbf{x}}(t)\right).
\end{gather*}
\end{proposition}

\begin{proof}

Assuming that the $(Z_i^{\mathbf{x}(t)})_i$ are i.i.d., Cram\'er's theorem states that
the LDP rate function for their average is
 $I(x)=\Lambda_{Z^{\mathbf{x}(t)}}^*(x)=\underset{\alpha \in \R^{D+3}}{\sup}\left\{ \left\langle \alpha, x\right\rangle - \Lambda_{Z^{\mathbf{x}(t)}}(\alpha)\right\}$, with $\Lambda_{Z^{\mathbf{x}(t)}}(\alpha)=\log\mathbb{E}\left[e^{\left\langle \alpha, Z^{\mathbf{x}(t)} \right\rangle}\right]$. 
 
In this case, with $\alpha=\left(\alpha_s, \alpha_u, \alpha_0, \dots, \alpha_D \right)$, we have
$
\Lambda_{Z^{\mathbf{x}(t)}}(\alpha)  = H\left(\mathbf{x}(t), \alpha\right),
$
being $H(x, \alpha)$ ($H: E \times \R^{D+3}\rightarrow \R$) the $\log$ of the moment-generating function of the (conditioned) Multinomial vector $Z^x$ (with $x\in E$) evaluated in $\alpha$, which is presented in Equation \eqref{eq:HX}.  Defining $L:E\times \R^{D+3}\rightarrow \R$ as in Equation \eqref{eq:LX}, results that $L\left(\mathbf{x}(t), \dot{\mathbf{x}}(t)\right)=\underset{\alpha}{\sup} \left\{\left\langle \alpha, \dot{\mathbf{x}}(t)\right\rangle - H\left(\mathbf{x}(t), \alpha\right)\right\}$ coincides with $\Lambda^*_{Z^{\mathbf{x}(t)}}\left(\dot{\mathbf{x}}(t)\right)$. 
\end{proof}

That is, the global cost of a deviation of $\left\{X_t^N \right\}_t$ to a trajectory $\mathbf{x}(t)$ can be interpreted as a consequence of the accumulated cost of microscopic deviations of the average of (conditioned) Multinomial random vectors, representing the degrees of the new explored vertices in one step.

\begin{remark}[Fluid limit]
Observe that, in particular, the mean macroscopic behavior $\mathbf{x}(t)$ should verify:
\begin{align*}
\dot{\mathbf{x}}(t) & \approx \mathbb{E}\left(Z^{\mathbf{x}(t)}\right) = \left(1,0,-1,0, \dots, 0\right)\frac{e_0(t)}{\sum_j e_j(t)}\\
& + \sum_{k=1}^D \left(1,-2k,0, -k\frac{e_1(t)}{u(t)}, \dots, -k\frac{je_j(t)}{u(t)}, \dots, -1-k\frac{ke_k(t)}{u(t)}, \dots, -k\frac{De_D(t)}{u(t)}\right) \frac{e_k(t)}{\sum_j e_j(t)},
\end{align*}
which coincides with the fluid limit that we formally prove in next subsection.
\end{remark}

\subsection{Fluid limit of the process $\left\{X_t^N\right\}_t$} \label{subsection:2}
In this subsection, we formally deduce the fluid limit of the process $\left\{X_t^N\right\}_t$.

\begin{proposition}[Fluid limit]
The sequence of processes $\left\{X^N\right\}_N$  converges almost-sure, as $N\rightarrow \infty$, to the deterministic function $\hat{x}:[0,1]\rightarrow E$ given by
\begin{gather*}
\hat{x}(t)=\begin{cases} \left(s(t), \hat{u}(t), \hat{e}_0(t),\dots, \hat{e}_D(t)\right), & \text{ if } t\leq T^*,\\ \left(T^*, 0,\dots, 0\right), & \text{ if } t>T^*, \end{cases} \text{ where }
\hat{e}_i(t)=\begin{cases} e_i(t), & \text{ if } t\leq t_i,\\0, & \text{ if } t>t_i.\end{cases}
\end{gather*}
The times $t_i$ are defined by $t_i = \inf \left\{ t\in [0,1]: e_i(t_i)\leq 0\right\}$ and  $x(t)=\left(s(t), u(t), e_0(t), \dots, e_D(t)\right)$  is (the) solution of the following ordinary differential equation:
\begin{equation} \label{eq: fluid limit}
\begin{cases}
\dot{s}=1,\\
\dot{u}=\frac{-2 \sum_k ke_k}{ \sum_k e_k},\\
\dot{e}_i= \frac{-e_i-\frac{ie_i}{u} \sum_k k e_k}{\sum_k e_k}, \quad i=0,\dots, D,\\
s(0)=0, \, u(0)=\lambda, \, e_i(0)=p_i.
\end{cases}
\end{equation}
$\hat{u}$ is the solution of Equation \eqref{eq: fluid limit} replacing $e_i$ by $\hat{e}_i$ and $T^* = \inf \left\{t\in [0,1]: \, \sum_k \hat{e}_k(t)=0\right\}= \max \left\{t_1, \dots, t_D\right\}$.
\end{proposition}


\begin{proof}
The cost function $L(x, \beta)$ defined in Theorem \ref{thm:LDPX^N} satisfies $L(x, \beta)\geq 0$ and  $L\left(x, \beta \right)=0$ if and only if $\beta = H_{\alpha} \left(x,0\right)$, where $H_{\alpha}\left(x, \alpha\right)$ are the partial derivatives of $H\left(x, \alpha\right)$ w.r.t. $\alpha=\left(\alpha_s,\alpha_u, \alpha_0, \dots, \alpha_D \right)$. Then, the trajectories with zero cost are the ones that verify 
$\dot{x}=H_{\alpha}\left(x,0\right)$. If in addition we impose the condition $x(0)=\left(0, \lambda, p_0, \dots, p_D\right)$, we obtain the autonomous Equation \eqref{eq: fluid limit}. Cauchy-Peano existence theorem ensures the existence of at least one solution of such equation. Let $\mathcal{D}=\left\{ x\in E: \, e_i>0 \, \forall i \right\}$ and $f(x)=H_{\alpha}\left(x,0\right)$. Then $f$ is a $C^1$-function on $\mathcal{D}$, i.e. it is a locally Lipschitz continuous function on $\mathcal{D}$. This implies the uniqueness of solutions $e_i(t)$ for equation $\begin{cases} \dot{x}=f(x), \\ x(0)=x_0 \in \mathcal{D},\end{cases}$ until the time $t_i$ at which $e_i(t_i)=0$, and then we take the solution $e_i(t)=0$ for all $t\geq t_i$. 
\end{proof}

\subsection{Optimization of the rate function.} \label{subsection:3}
The following proposition allows to transform the optimization problem of the rate function $I$ over a set of trajectories into a much simpler optimization problem on $\mathbb R$.

\begin{proposition}[Rate function optimization] \label{thm:Rate Function Optimization}
Let $A$ be a subset of $D_{E}[0,1]$. Then,
\begin{gather*}
\underset{x\in A}{\inf} I(x)= \underset{\left\{\alpha_0\in \R^{D+3}: \hat{x}_{\alpha_0}\in \bar{A}\right\}}{\inf} I\left(\hat{x}_{\alpha_0}\right),
\end{gather*}
where the closure of $A$ is considered w.r.t. the Skorohod topology, \\
$\hat{x}_{\alpha_0}(t)=\begin{cases} \left(s_{\alpha_0}(t), \hat{u}_{\alpha_0}(t), \hat{e}_{0, \alpha_0}(t),\dots, \hat{e}_{D, \alpha_0}(t)\right), & \text{ if } t\leq T_{\alpha_0}\\ \left(T_{\alpha_0}, 0,\dots, 0\right), & \text{ if } t>T_{\alpha_0}\end{cases}$,\\
$\hat{e}_{i, \alpha_0}(t)=\begin{cases} e_{i, \alpha_0}(t), & \text{ if } t\leq t_{i, \alpha_0},\\0, & \text{ if } t>t_{i, \alpha_0}\end{cases}$,\quad\quad $t_{i, \alpha_0} = \inf \left\{ t\in [0,1]: e_{i, \alpha_0}(t)\leq 0\right\}$, and \\
$x_{\alpha_0}(t)= \left(s_{\alpha_0}(t), u_{\alpha_0}(t), e_{0,\alpha_0}(t), \dots, e_{D,\alpha_0}(t)\right)$ is (the) solution of the ODE:
\begin{equation}\label{eq:optima Hamilton}
\begin{cases}
\dot{x}= H_{\alpha}(x, \alpha);\\
\dot{\alpha}= -H_x(x, \alpha);\\
x(0)=\left(0,\lambda, p_0,\dots, p_D\right); \, \alpha(0)=\alpha_0.
\end{cases}
\end{equation}
$H_x$ and $H_{\alpha}$ are the vectors of partial derivatives of $H$ w.r.t. $x$ and $\alpha$, \\ and $T_{\alpha_0}=\inf \left\{t \in [0,1]: \sum_k \hat{e}_{k,\alpha_0}(t)=0\right\}$. 
\end{proposition}

\begin{remark} As expected, for $\alpha_0 = \left(0, 0, \dots, 0\right)$, $\hat{x}_{\alpha_0}(t)$ coincides with the fluid limit, which is solution of Equation \eqref{eq: fluid limit}, and $\alpha(t)=\left(0, \dots, 0\right)$ for all $t$. Then $\underset{x\in A}{\inf} I(x)=0$ if the fluid limit belongs to $A$.
\end{remark}

\begin{proof}
Note that if $\mathbf{x} \in \mathcal{H}_L$ is such that $\mathbf{x}(t)=\left(s(t),u(t), e_0(t), \dots, e_D(t)\right)$ and $\sum_k e_k(t)=0$ for all $t\geq t_0$, then $I(\mathbf{x})=\intop_{0}^{1}L(\mathbf{x}, \dot{\mathbf{x}})\text{d}t= \int_{0}^{t_0}L(\mathbf{x}, \dot{\mathbf{x}})\text{d}t$, so just consider Hamilton's equations for the case $\sum_k e_k >0$. Hamilton's equations, presented in Equation \eqref{eq:optima Hamilton}, give conditions for a function $\mathbf{x}$ to be a \emph{stationary curve} of the functional $I$ (and are equivalents to Euler-Lagrange equation, see \cite{Arnold}, for example). Note that $\alpha$ is an auxiliary function. 
\end{proof}

\subsection{LD for the size of the independent set constructed by the algorithm}  \label{subsection: 4}
We can now deduce LD results for the sequence of stopping times $\frac{T_N^*}{N}$, (which coincide with the proportion of vertices in the independent set constructed by the algorithm) using our previous results:

\begin{theorem}\label{LDP_T}
Consider $T_N^*$ defined before as the stopping time of the algorithm presented in Section \ref{section: Description}. 
\begin{enumerate}
\item If $\varepsilon>0$ is such that $T^*+\varepsilon<1$, then 
\begin{gather*}
\underset{N}{\lim} \frac{1}{N} \log \mathbb{P} \left(\frac{T^*_N}{N} \geq T^*+\varepsilon \right) = -F^+\left( T^* + \varepsilon\right),
\end{gather*}
being $F^+\left( T^* + \varepsilon\right)= \inf\left\{ I\left(\hat{x}_{\alpha_0}\right): \; T_{\alpha_0}\geq T^*+\varepsilon , \alpha_0\in \R^{D+3}\right\}$.

\item If $\varepsilon>0$ is such that $T^*-\varepsilon>0$, then 
\begin{gather*}
\underset{N}{\lim} \frac{1}{N} \log \mathbb{P} \left(\frac{T^*_N}{N} \leq T^*-\varepsilon \right) = -F^-\left( T^* - \varepsilon\right),
\end{gather*}
being $F^-\left( T^* - \varepsilon\right)= \inf\left\{ I\left(\hat{x}_{\alpha_0}\right): \; T_{\alpha_0}\leq T^*-\varepsilon , \alpha_0\in \R^{D+3}\right\}$.
\end{enumerate}
In both cases $\hat{x}_{\alpha_0}$ and $T_{\alpha_0}$ are as in Proposition \ref{thm:Rate Function Optimization}.
\end{theorem}

\begin{proof} We only prove the first statement because the proof of the second one is analogous. Define the set $A_{\varepsilon}$, that contains the trajectories $\mathbf{x}\in D_E[0,1]$ such that $\mathbf{x}(t)=\left(s(t), u(t), e_0(t), \dots, e_D(t)\right)$, $\mathbf{x}(0)=\left(0, \lambda, p_0, \dots, p_d\right)$, coordinates $e_i(t)$, $u(t)$ are decreasing, $s(t)$ is increasing, $0\leq e_i(t), s(t)\leq 1$ for all $t$, and such that $\inf\left\{t: \, \sum_k e_k(t)=0\right\}\geq T^*+\varepsilon$. Then, Proposition \ref{thm:Rate Function Optimization} implies that
\begin{align*}
\lim_{N} \frac{1}{N} \log \mathbb{P} \left(\frac{T^*_N}{N} \geq T^*+\varepsilon \right) & = \lim_N \frac{1}{N} \log \mathbb{P} \left(X_.^N\in A_{\varepsilon} \right)\\
& = -\inf_{\left\{\alpha_0 \in \R^{D+3} : \, \hat{x}_{\alpha_0}\in \bar{A}_{\varepsilon}\right\}} I\left(\hat{x}_{\alpha_0}\right)
= F^+\left( T^* + \varepsilon\right).
\end{align*}
\end{proof} 
\section{$d$-regular case}\label{section: d-regular}
In this section, we analyze the results presented in previous sections for the particular case of a $d$-regular graph, i.e.  $p_d = 1$ and $p_i = 0$ for all $i\neq d$. In this case, the sequence of interest is $\left\{X_t^N\right\}_{t\in [0,1]}$ with $X_t^N=\frac{1}{N}X_{[Nt]}^N$, being $X_n^N= \left(S_n^N, U_n^N, E_n^N\right)$, and
\begin{itemize}
\item $S_n^N=\left|\mathcal{S}_n^N\right|$, the number of vertices that have already been assigned to the independent set at step $n$;
\item $U_n^N$, the total number of unpaired half-edges at step $n$;
\item $E_n^N=\left|\mathcal{E}_n^N\right|$, the number of empty vertices at step $n$.
\end{itemize}
$X_t^N\in E^N$, being 
$E^N= \left\{ \frac{1}{N}(\hat{s}, \hat{u}, \hat{e}):\, \hat{s}, \hat{e}\in \{0,\dots, N\}; \, \hat{u}\in \{0,\dots, dN\};\, \hat{u}\geq d\hat{e}\right\}$, which is a subset of the compact set 
$E:=\left\{(x_1,x_2, x_3) \in [0,1]\times[0,d]\times[0,1]:\, x_2\geq dx_3\right\} \subset \R^3$. The Hamiltonian $H\colon E\times \R^3 \rightarrow \R$ is given by
\begin{equation}
H\left(x, \alpha\right)= 
\begin{cases} 
\alpha_1-2d\alpha_2-\alpha_3+d\log\left[1+\left(e^{-\alpha_3}-1\right)\frac{dx_3}{x_2}\right], &\text{ if } x_3>0, \\ 
0, & \text{ if } x_3=0, 
\end{cases}
\label{eq:H}
\end{equation}
where $x=(x_1,x_2,x_3)$, $\alpha= (\alpha_1, \alpha_2, \alpha_3)$, and the cost function $L$ can be obtained explicitly as $L\colon E\times \R^3\rightarrow \R$ such that
\begin{equation} 
L(x, \beta)= 
\begin{cases}
(\beta_3+1)\alpha_3^*-d\log \left[1+\left(e^{-\alpha_3^*}-1\right)\frac{dx_3}{x_2}\right],
\text{ with } \alpha_3^*=\log\left[\frac{dx_3}{dx_3-x_2}\left(\frac{d}{\beta_3+1}+1\right)\right], \\ 
\qquad \qquad \qquad \qquad \qquad \qquad \qquad \qquad \qquad \qquad  \text{ if } \beta_1=1, \beta_2= -2d, \beta_3 \geq -(d+1),\\
0,  \text{ if } x_3=\beta_3=0,\\
+\infty,  \text{ in other cases.}
\end{cases}
\end{equation}

\subsection{Fluid limit}
The trajectories with zero cost are $\mathbf{x}(t)=\left(s(t), u(t), e(t)\right)$ such that $\dot{s}=1$, $\dot{u}=-2d$, and $\dot{e}=-1-\frac{d^2e}{u}$. For the initial condition $\mathbf{x}(0)=(0,d,1)$ with $d\geq 3$, the unique solution is $\mathbf{x}(t)=\left(s(t), u(t), e(t)\right)$ with $s(t)=t$, $u(t)= d(-2t+1)$, and 
\begin{gather*}
e(t)= \begin{cases} \frac{1}{d-2} \left[2t-1+(d-1)(1-2t)^{\frac{d}{2}}\right], & \text{ if } t\leq T^*,\\
0, &\text{ if } t>T^*,
\end{cases}
\end{gather*}
where $T^{*}=\inf \left\{t: e(t)=0\right\}$, i.e. the \emph{jamming constant} is $T^{*}=\frac{1}{2} \left[ 1- \left(\frac{1}{d-1}\right)^{\frac{2}{d-2}} \right]$. This coincides with the known result given in \cite{wormaldDF}.

If $d=2$, then the fluid limit is $\mathbf{x}(t)=\left(s(t), u(t), e(t)\right)$ with the same functions $s(t)$ and $u(t)$, and $e(t)$ is given by
\begin{gather*}
e(t)= \begin{cases} (1-2t)\left[\frac{1}{2}\log(1-2t)+1\right], & \text{ if } t\leq T^*,\\
0, &\text{ if } t>T^*,\end{cases}
\end{gather*}
where $T^*= \frac{1-e^{-2}}{2}$. This coincides with the known result from the earlier work of \cite{Flory}.
\subsection{LDP for $\left\{E^N\right\}_N$} \label{subsection: 2}
Since the trajectories with positive probability for $X^N_t$, when $N \rightarrow \infty$, are those $\mathbf{x}(t)=\left(s(t), u(t), e(t)\right)$ such that $s(t)=t$ and $u(t)=d(-2t+1)$, we can directly deduce an LDP for a rescale of the process that counts the number of unexplored vertices in each step of the algorithm.

\begin{corollary}[LDP for $\left\{E^N\right\}_N$] \label{thm:LDP_EN} 
The sequence of processes $\left\{E^N\right\}_N$ given by $E_t^N :=\frac{E_{[Nt]}^N}{N}$ verifies an LDP in $D_{[0,1]}[0,1]$ with good rate function $\hat{I}:D_{[0,1]}[0,1] \rightarrow \left[0, +\infty\right]$ such that
$\hat{I}(x)= \int_0^1 \hat{L}\left(t, x(t), \dot{x}(t)\right)dt$, where
\begin{gather*}
\hat{L}(t,x,y) = L\left(\left(t, u(t), x\right), \left(1, -2d, y \right)\right), \text{ with } u(t)= d(-2t+1).
\end{gather*}
\end{corollary}

\begin{proof}
It is deduced directly from Theorem \ref{thm:LDPX^N}.
\end{proof}

\begin{remark} The cost function $\hat{L}(t,x,y)$ verifies $\hat{L}(t,x,y)= \Lambda^*_{W_{t,x}}(y)$, being $\Lambda^*_{W_{t,x}}(y)$ the LD rate function for the average of random variables $W_{t,x}=B_{t,x}-d-1$, where $B_{t,x}$ has a Binomial distribution with parameters $n=d$ and $p=1-\frac{dx}{u(t)}$. This corresponds to the approximation to the distribution of new explored vertices in each step of the algorithm.
Note that $\Lambda^*_{W_{t,x}}(y)=\Lambda^*_{B_{t,x}}(y+d+1)$. 
\end{remark}
\bigskip
In this case, we can explicitly deduce the LD rate for the size of the independent set constructed by the algorithm since the optimization problem of the rate function $\hat{I}$ over a set of trajectories of $D_{[0,1]}[0,1]$ becomes an optimization problem in $\R$.

\begin{corollary}[Optimization of the rate function $\hat{I}$] \label{thm:d-Rate Function Optimization}
Let $A$ be a subset of $D_{[0,1]}[0,1]$. Then,
$
\underset{x\in A}{\inf} \hat{I}(x)= \underset{\left\{\alpha_0\in \R: \hat{x}_{\alpha_0}\in \bar{A}\right\}}{\inf} F(\alpha_0),
$
being $F(\alpha_0)=\int_0^{T_{\alpha_0}}\hat{L}\left(t, x_{\alpha_0}(t), \dot{x}_{\alpha_0}(t)\right)dt$, $x_{\alpha_0}$ is the solution of the ODE:
\begin{equation}\label{eq:d-optima Hamilton}
\begin{cases}
\dot{x}= -1+\frac{dx}{e^y(2t-1+x)-x};\\
\dot{y}= \frac{d(1-e^y)}{e^y(2t-1+x)-x};\\
x(0)=1; \, y(0)=\alpha_0,
\end{cases}
\end{equation}
$T_{\alpha_0}=\inf \left\{t \in [0,1]: x_{\alpha_0}\leq 0\right\}$, and $\hat{x}_{\alpha_0}(t)=\begin{cases} x_{\alpha_0}(t), & \text{ if } 0\leq t\leq T_{\alpha_0},\\ 0, & \text{ if } t>T_{\alpha_0}.\end{cases}$
\end{corollary}
\begin{proof}
It is a corollary of Proposition \ref{thm:Rate Function Optimization}.
\end{proof}



In Figure \ref{fig:F_alpha0} (left), the evolution of $F$ as a function of $\alpha_0$ is presented for $d=2,\dots, 10$. It is observed that when $d=2$ (where the problem basically boils down to finding independent sets on circles), the possibility of deviating from the fluid limit is much higher than for larger $d$. 

\begin{figure}
\begin{tabular}{cc}
\includegraphics[scale=0.35]{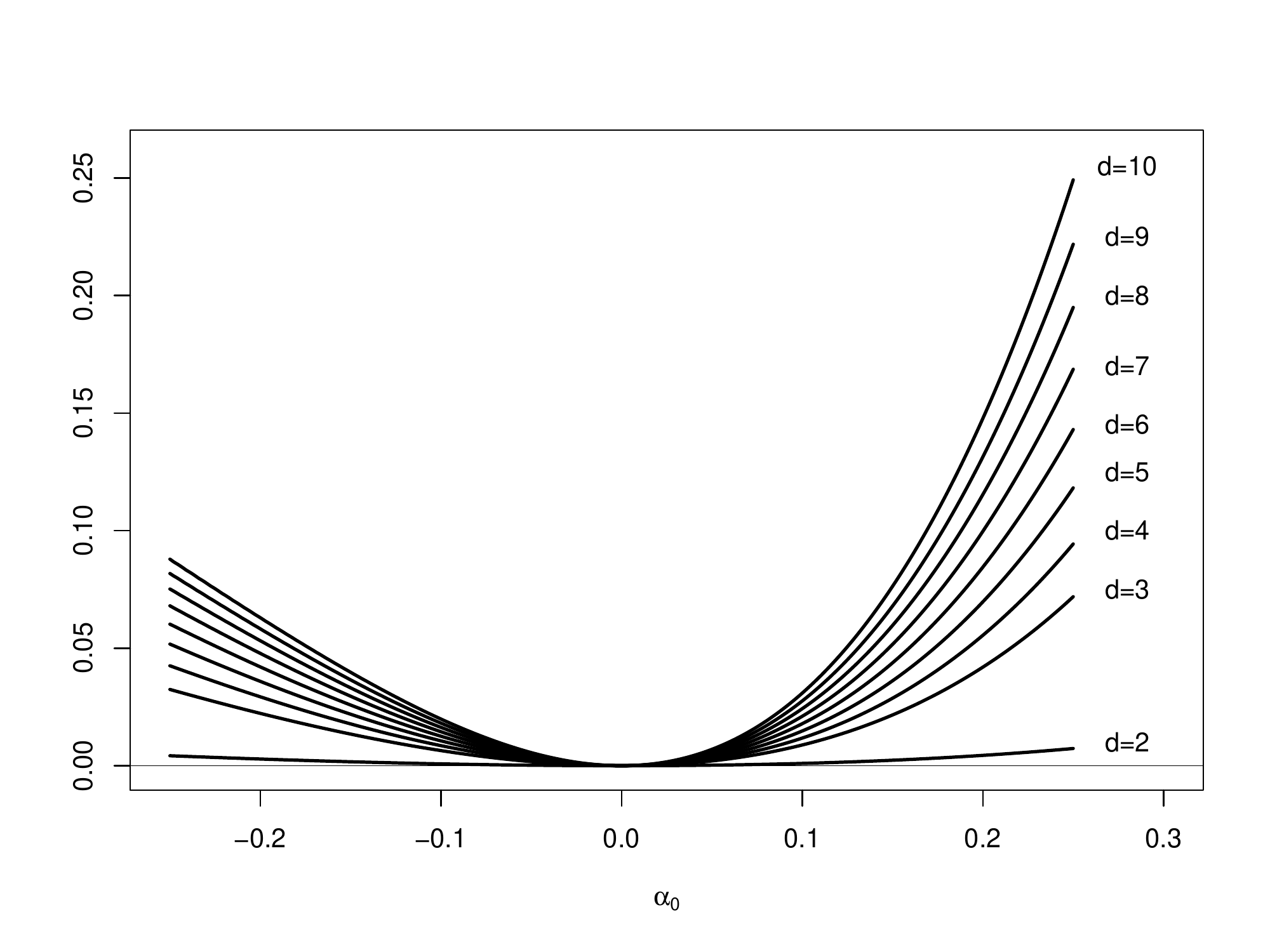} 
\includegraphics[scale=0.35]{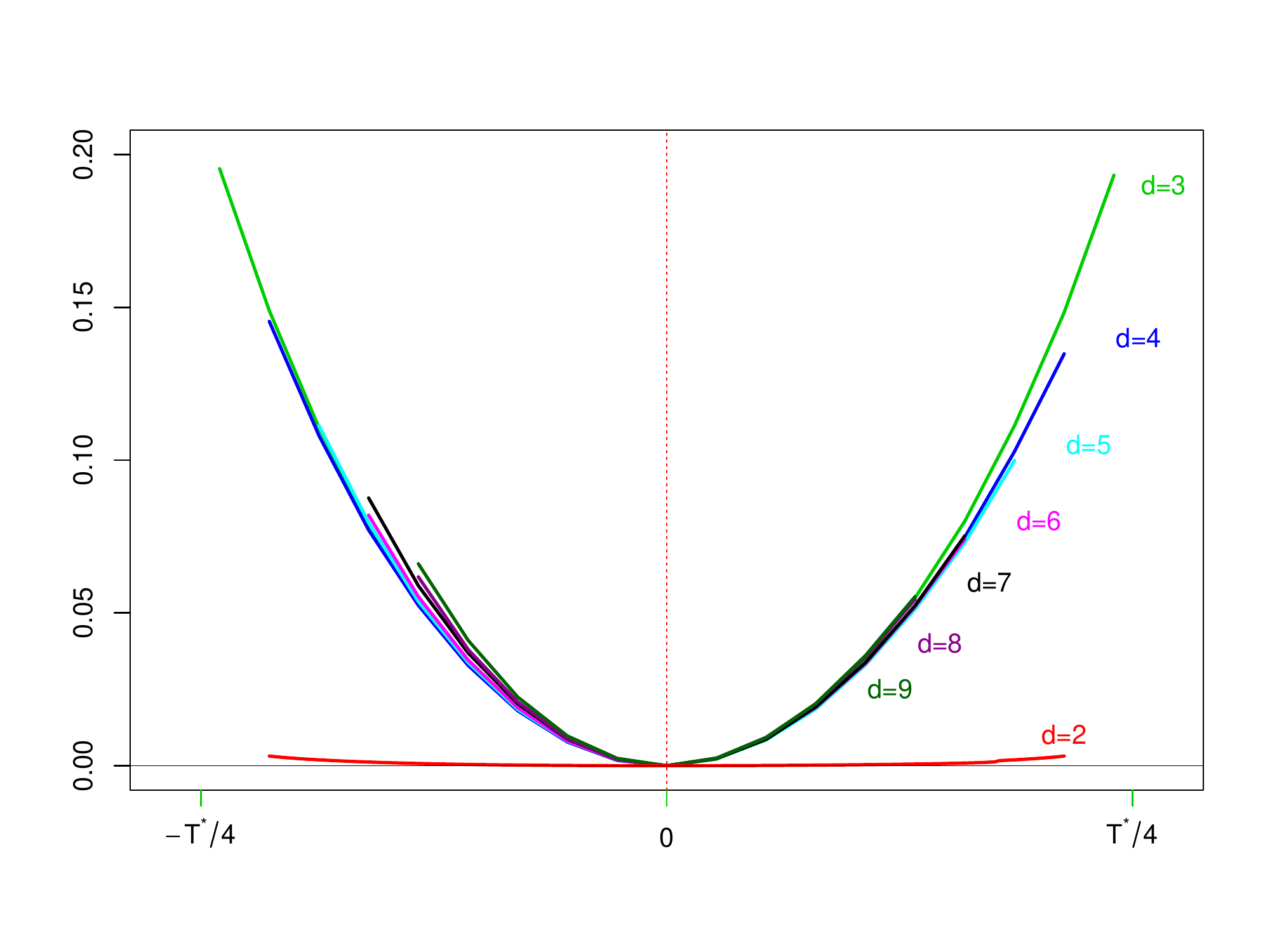}
\end{tabular}

\caption{\emph{Left:} Evolution of $F(\alpha_0)$ as function of $\alpha_0$ for $d=2, \dots, 10$. 
\emph{Right:} Evolution of $F\left(\alpha_0 (T^* \pm \varepsilon) \right)$ as function of $\varepsilon$ for $d=2, \dots, 9$.
}\label{fig:F_alpha0}
\end{figure}

\begin{corollary}
The sequence of stopping times $T_N^*$ defined over a $d$-regular graph $G(N,d)$ verifies: 
\begin{enumerate}
\item If $\varepsilon>0$ is such that $T^*+\varepsilon<1$, then $$\underset{N}{\lim} \frac{1}{N} \log \mathbb{P} \left(\frac{T^*_N}{N} \geq T^*+\varepsilon \right) = -F\left(\alpha_0 (T^* + \varepsilon)\right),$$ where $\alpha_0(T^*+\varepsilon)$  is the unique real number $\alpha_0>0$ such that $T_{\alpha_0} = T^*+\varepsilon$.
\item If $\varepsilon>0$ is such that $T^*-\varepsilon>0$, then $$\underset{N}{\lim} \frac{1}{N} \log \mathbb{P} \left(\frac{T^*_N}{N} \leq T^*-\varepsilon \right) = -F\left(\alpha_0 (T^* - \varepsilon)\right),$$ where $\alpha_0(T^*-\varepsilon)$ is the unique real number $\alpha_0<0$ such that $T_{\alpha_0} = T^*-\varepsilon$.
\end{enumerate}
In both cases $F(\alpha_0)$ and $T_{\alpha_0}$ are as in Proposition \ref{thm:d-Rate Function Optimization}.
\end{corollary}

Figure \ref{fig:F_alpha0} (right) presents the evolution of $F\left(\alpha_0 (T^* \pm \varepsilon)\right)$ as a function of $\varepsilon \in \left[0, \frac{T^*}{4}\right]$ for $d=3, \dots, 9$, compared with $\varepsilon \in \left[0, \frac{T^*}{6}\right]$ for $d=2$. Note that in each case the time $T^*$ depends on $d$. Again, the abrupt change in the dynamics is observed for $d=2$ and $d>2$. 

\begin{proof} In this case, if $x_{\alpha_0}$ is the solution of Equation \eqref{eq:d-optima Hamilton} with $y(0)=\alpha_0$, then the following monotony property with respect to the initial condition $\alpha_0$ is verified:
$
\text{ if } \alpha_0 <\alpha_1 \Rightarrow x_{\alpha_0}(t)<x_{\alpha_1}(t) 
$
for all $t$, then $T_{\alpha_0} <T_{\alpha_1}$. In addition, it can be seen that for all $T\in\left(T^*,\, 1\right)$, there exists a unique value $\alpha_0= \alpha_0(T)>0$ such that $x_{\alpha_0}(T)=1$ (i.e. $T=T_{\alpha_0}$). Then, there is only one $\alpha_0^*>0$ such that $x_{\alpha_0^*}\left(T^*+\varepsilon\right)=1$. If $A_{\varepsilon}$ is the set defined in the proof of Theorem \ref{LDP_T}, then $\hat{x}_{\alpha_0} \in A_{\varepsilon} \Leftrightarrow \alpha_0 \geq \alpha_0^*$,
which implies that
$
 \underset{\{\alpha_0: \hat{x}_{\alpha_0}\in \bar{A}_{\varepsilon}\}}{\inf} F\left(\alpha_0\right) = \underset{\{\alpha_0 \geq \alpha_0^*\}}{\inf} F\left(\alpha_0\right)=F\left(\alpha_0^*\right).
$
\end{proof} 

\section{Proof of Theorem \ref{thm:LDPX^N}} \label{Proof_LDP}

The proof of Theorem \ref{thm:LDPX^N} goes similar to the one included with full details in \cite{bermolen2020} for the process that counts the number of unexplored vertices of an exploration algorithm over Erd\"os-Renyi graphs. Basically, it consists in verifying steps 1,2,3, 4, and using Theorems \textbf{5.15}, \textbf{8.14}, and \textbf{8.23} from \cite{F&K}. We briefly describe each of these steps below.  

The first step consists of proving the convergence of non-linear generators $\mathbf{H}_N(f)(x)=\log \left[e^{-Nf(x)}T^N\left(e^{Nf(x)}\right)\right]$, where $T^N$ is the linear generator of $\left\{\frac{X_n^N}{N}\right\}_n$, and derive the limit operator $\mathbf{H}$. In our case, using Taylor's theorem and Stirling's formula we prove that $\mathbf{H}_N$ converges to $\mathbf{H}$ in the following sense: $\underset{N\rightarrow \infty}{\lim}      \underset{x\in E^N}{\sup} \left|\mathbf{H}_N(f)(x)-\mathbf{H}(f)(x)\right|=0$ for all $f\in C^1(E)$, being $\mathbf{H}: C^1(E)\rightarrow \R$ such that $\mathbf{H}(f)(x)=H\left(x, \nabla f(x)\right)$, and $H$ is defined in Equation \eqref{eq:HX}. 

The second step consists of verifying the \emph{exponential compact containment condition} (see Condition \textbf{2.8} from \cite{F&K}). In our case, it is trivially verified since the state space $E$ is a compact subset of $\R^{D+3}$.

The third step consists of proving that $\mathbf{H}$ generates a semigroup $\mathbf{V}= \left\{V_t\right\}_t$. This issue is nontrivial and follows by showing that the Hamilton-Jacobi equation:
\begin{equation}\label{eq:Hamilton-Jacobi}
f(x)-\beta H \left(x,\nabla f(x)\right)-h(x)=0,    
\end{equation}
has a unique solution $f$ for all $h\in C(E)$ and $\beta>0$ in a viscosity sense. 
In our case, we use results from \cite{Kraaij} to prove that this Hamilton-Jacobi equation verifies the \emph{comparison principle}, which ensures the uniqueness of the viscosity solution, explicitly constructed in Chapter 9 of \cite{F&K}. 
As a consequence of Proposition {\bf 4.2} in \cite{Kraaij}, it is enough to prove that the following inequality holds:
\begin{equation}\label{eq:Condition-Kraaij}
\liminf_{\alpha\rightarrow \infty} H\left( x^{\alpha}, \alpha \psi_x\left(x^{\alpha}, y^{\alpha}\right)\right) -  H\left( y^{\alpha}, \alpha \psi_x\left(x^{\alpha}, y^{\alpha}\right)\right) \leq 0,   
\end{equation}
where $\psi_x(x,y)$ is the vector of partial derivatives w.r.t. $x=\left(s, u, e_0, \dots, e_D\right)$ of the good distance function $\psi(x,y)= \frac{1}{2}\left\Vert x-y\right\Vert^2$. $x^{\alpha}=\left(s^{x_{\alpha}}, u^{x_{\alpha}}, e^{x_{\alpha}}_0, \dots, e^{x_{\alpha}}_D\right) $ and $y^{\alpha}=\left(s^{y_{\alpha}}, u^{y_{\alpha}}, e^{y_{\alpha}}_0, \dots, e^{y_{\alpha}}_D\right)$ are the sequences constructed in Chapter 9 of \cite{F&K} (with $\alpha\rightarrow +\infty$) and verify $\left( x^{\alpha}, y^{\alpha}\right) \rightarrow (z,z)$ where $\mu(z)-v(z)=\underset{x \in E}{\sup} \left\{ \mu(x)- v(x)\right\}$ for a given subsolution $\mu$ and supersolution $v$ of Equation \eqref{eq:Hamilton-Jacobi}.

Finally, the limiting semigroup $\mathbf{V}=\left\{V_t\right\}_t$ usually admits a variational form $\left\{\mathcal{V}_t\right\}_t$, known as the \emph{Nisio semigroup} in control theory. Then, the fourth step consists of providing the more treatable form of the rate function $I$ presented in Theorem \ref{thm:LDPX^N}.
In our case, as $\mathbf{H}(f)(x)= H \left(x, \nabla f(x)\right)$ for each $x\in E$ and $H \leftrightarrow L$ \footnote{We use $H \leftrightarrow L$ to denote that $L(x, \beta)= \underset{\alpha}{\sup} \left\{ \left\langle \alpha, \beta\right\rangle - H\left( x, \alpha\right)\right\}$ and $H(x, \alpha)= \underset{\beta}{\sup} \left\{ \left\langle \alpha, \beta\right\rangle - L\left( x, \beta\right)\right\}$ }, we have that $\mathbf{H}$ can be written as 
$
\mathbf{H}(f)(x)= \underset{u \in U}{\sup} \left\{ A(f)(x,u)-L(x,u) \right\}, 
$
where $U= \R^{D+3}$ and $A:C^1(E)\rightarrow M\left(E\times U\right)$ is the linear operator given by $A(f)(x,u)=\left\langle \nabla f(x), u\right\rangle$. 
Then, we consider the Nisio semigroup corresponding to the control problem determined by $A$ and the cost function $-L$:
\begin{gather}\label{eq:Vt}
\mathcal{V}_t (f)(x_0)= \sup_{\{\left(\mathbf{x},\lambda \right) \in \mathcal{Y}: \; \mathbf{x}(0)=x_0\}} 
\left\{f(\mathbf{x}(t)) - \iint_{U\times [0,t]}  L(\mathbf{x}(s), u) \lambda\left(\text{d}u\times \text{d}s\right) \right\}
\end{gather}
for each $x_0 \in E$. $\mathcal{Y}\subset D_E[0,1] \times \mathcal{M}_m(U)$ is the \emph{control set} of the linear operator $A$ (see Definition \textbf{8.1} from \cite{F&K}), and $\mathcal{M}_m(U)$ is the space of Borel measures $\lambda$ on $U \times [0,1]$ satisfying $\lambda\left( U\times [0,t]\right)=t$ for all $t\in [0,1]$. Measure $\lambda$ is known as a \emph{relaxed control}. As $L$ is convex w.r.t. $\beta$, it follows that a deterministic control $\lambda\left(\text{d}u\times \text{d}s\right)=\delta_{u(s)}(\text{d}u)\text{d}s$ is always the control with smallest cost by Jensen's inequality, and the supremum in Equation \eqref{eq:Vt} is attained on $\mathcal{Y}_{\mathcal{AC}}:=\left\{ (\mathbf{x}, \lambda) \in \mathcal{Y}: \, \mathbf{x} \in \mathcal{AC}, \; \mathbf{x}(0)=x_0 \right\}$, being $\mathcal{AC}$ the space of absolutely continuous functions $\mathbf{x}: E \rightarrow \R$.\\
Then, as consequence of Theorems \textbf{8.14} and \textbf{8.23} from \cite{F&K}, it is enough to prove that Conditions \textbf{8.9}, \textbf{8.10} and \textbf{8.11} from \cite{F&K} are verified to prove that $V_t= \mathcal{V}_t$ and $I$ can be written as in Theorem \ref{thm:LDPX^N}.\\
To prove these conditions, we consider elements $\left(\mathbf{x}, \lambda\right)$ from $\mathcal{Y}_{\mathcal{AC}}$ and use that $L\left(x, q(x)\right)=0$ for all $x\in E$ if $q(x)=H_{\alpha}\left(x,0\right)$, and $H\left(x, \nabla f(x)\right)=\left\langle \nabla f(x), q_f(x)\right\rangle - L\left(x, q_f(x)\right)$ if $q_f(x)=H_{\alpha}\left(x, \nabla f(x)\right)$.

In a nutshell, as a consequence of the first two steps, the process verifies the exponential tightness condition; the third step assures the existence of an LDP, and the fourth step provides the useful version of the rate presented in this theorem. 

\bibliographystyle{amsplain}
\bibliography{References.bib}








\end{document}